\documentclass[11pt]{article}
\usepackage[margin=1in]{geometry}
\usepackage{amsmath,amssymb,amsthm,mathtools}
\usepackage{enumitem}
\usepackage{microtype}
\usepackage{hyperref}
\usepackage[nameinlink,capitalize]{cleveref}

\newtheorem{theorem}{Theorem}[section]
\newtheorem{proposition}[theorem]{Proposition}
\newtheorem{lemma}[theorem]{Lemma}
\newtheorem{corollary}[theorem]{Corollary}
\newtheorem{remark}[theorem]{Remark}
\theoremstyle{definition}
\newtheorem{definition}[theorem]{Definition}

\newcommand{\Cay}{\operatorname{Cay}}
\newcommand{\AGL}{\operatorname{AGL}}
\newcommand{\Spec}{\operatorname{Spec}}
\newcommand{\CP}{\operatorname{CP}}

\title{Ramanujan Cayley Graphs with Normal Connection Sets\\
in Ratio-One Frobenius Groups}
\author{Ming-Hsuan Kang \and Chi-Jung Yang}
\date{}

\hypersetup{
  hidelinks,
  pdftitle={Ramanujan Cayley Graphs with Normal Connection Sets in Ratio-One Frobenius Groups},
  pdfauthor={Ming-Hsuan Kang and Chi-Jung Yang},
  pdfsubject={Combinatorics},
  pdfkeywords={Ramanujan graph, Cayley graph, Frobenius group, sharply two-transitive group, lexicographic product, graph spectrum}
}

\begin{document}
\maketitle

\begin{abstract}
Let $G=N\rtimes H$ be a finite Frobenius group with $|N|=q$ and
$|H|=q-1$.  We classify all Ramanujan Cayley graphs of $G$ whose
connection sets are normal, in the sense of being unions of conjugacy
classes.  The group-theoretic input is a simple blow-up phenomenon:
every such Cayley graph is either $Y[\overline{K_q}]$ or $Y[K_q]$ for
a connected regular Cayley graph $Y$ on the complement $H$.  We first
prove a graph-theoretic result classifying all Ramanujan graphs of
these two forms when $Y$ is an arbitrary connected regular graph on
$q-1$ vertices.  The proof combines the classical characterization of
regular graphs with least eigenvalue greater than $-2$ with a
second-moment identity in the bipartite case.  Translating the resulting
five graph types back to $G$ yields a complete classification for all
ratio-one Frobenius groups, and in particular for $\AGL(1,q)$ over every
finite field.
\end{abstract}



\section{Introduction}

A connected $d$-regular graph is Ramanujan if every adjacency
eigenvalue $\lambda$ with $|\lambda|\ne d$ satisfies
\[
 |\lambda|\le 2\sqrt{d-1}.
\]
Thus, for a bipartite graph, both $d$ and $-d$ are regarded as trivial
eigenvalues.  For a finite group $G$ and an inverse-closed generating
set $S\subset G\setminus\{1\}$, recall that the Cayley graph
$\Cay(G,S)$ has vertex set $G$ and an edge between $x$ and $y$
whenever $x^{-1}y\in S$; it is natural to ask how the structure
of $G$ constrains its Ramanujan property.

Hirano, Katata, and Yamasaki studied this question for Frobenius
groups --- groups $G=N\rtimes H$ in which the complement $H$ acts on
the kernel $N\setminus\{1\}$ without fixed points (\cref{def:frobenius}
below) --- when the connection set is \emph{normal}, meaning a union
of conjugacy classes \cite{HKY2016}.  Their central parameter is
\[
 r=\frac{|N|-1}{|H|}.
\]
They proved that a natural valency boundary is sharp for $r\ge4$ and
observed that small ratios behave differently; in particular, their
examples already show that additional bipartite Ramanujan graphs can
occur when $r\le3$.  The later survey of Liu and Zhou records this
large-ratio boundary theorem as the general Frobenius-group result
\cite{LiuZhou2022}.  The extreme case $r=1$ is maximally rigid: the
complement has order $|N|-1$ and acts regularly on
$N\setminus\{1\}$.  Equivalently, the natural Frobenius action is
sharply two-transitive: for every two ordered pairs of distinct
points there is a unique element of $G$ carrying one pair to the
other.  Finite sharply two-transitive groups are precisely affine
groups over finite nearfields (algebraic systems satisfying every
field axiom except one distributive law) \cite{GrundhoferHering2017};
the familiar example is
\[
 \AGL(1,q)=\mathbb F_q\rtimes\mathbb F_q^\times.
\]

The present work grew from an explicit study of normal connection sets
in $\AGL(1,p)$ for prime $p$.  In that setting the character table
reduces the spectrum to trigonometric sums on $\mathbb F_p^\times$.
Individual valencies can then be handled one at a time, but the
calculation does not reveal why only a few configurations can survive
the Ramanujan bound.

Our main observation is structural.  In ratio one, every normal
connection set is determined by just two pieces of data: whether it
meets the kernel $N$, and a single normal, inverse-closed subset of
the complement $H$.  This forces the resulting Cayley graph on $G$ to
be a lexicographic blow-up of the corresponding Cayley graph on $H$,
by an edgeless graph or by a complete graph on $q$ vertices according
to whether the kernel is included.
Lexicographic product decompositions of Cayley graphs have been studied
more generally \cite{PengWang2007}.  Corr and Praeger also give a
coset criterion for the independent-fiber quotient decomposition of a
connected Cayley graph \cite[Proposition~2.13]{CorrPraeger2015}.
Here the ratio-one conjugacy structure is stronger: it forces, for
every normal connection set, one of the two complementary fiber models
needed below.

This suggests separating the proof into a graph-theoretic part and a
group-theoretic part.  We first classify, for an arbitrary connected
regular graph $Y$ on $q-1$ vertices, when either of the above two
$q$-fold blow-ups is Ramanujan.  The non-bipartite cases reduce to the
classical theorem of Doob and Cvetkovi\'c that a connected regular graph
with least eigenvalue greater than $-2$ is a clique or an odd cycle
\cite{DoobCvetkovic1979}; see also
\cite[Corollary~2.3.22]{CvetkovicRowlinsonSimic2004}.  The bipartite
case is controlled by
$\operatorname{tr}(A^2)$, which leaves only complete bipartite graphs
and complete bipartite graphs with a perfect matching deleted.

We emphasize that ``normal connection set'' is used throughout in the
conjugacy-invariant sense of Hirano--Katata--Yamasaki.  This is distinct
from the notion of a normal edge-transitive Cayley graph studied, for
example, by Corr and Praeger \cite{CorrPraeger2015}.  There are also
other Ramanujan constructions on affine groups with non-normal
connection sets \cite{BellMinei2006}; our theorem is a classification
within the normal-connection-set family.

\section{A spectral classification of the two blow-ups}

For graphs $Y$ and $Z$, write $Y[Z]$ for their lexicographic product.
Let $\overline{K_q}$ denote the edgeless graph on $q$ vertices.  We
begin with the spectral formula that drives the paper.

\begin{lemma}\label{lem:blowupspectrum}
Let $q\ge3$ and let $Y$ be a connected $t$-regular graph on $q-1$
vertices, with
\[
 \Spec(Y)=\{t=\theta_0,\theta_1,\ldots,\theta_{q-2}\}.
\]
Then
\begin{align*}
 \Spec\bigl(Y[\overline{K_q}]\bigr)
 &=\{q\theta_i:0\le i\le q-2\}
   \cup\{0^{(q-1)^2}\},\\
 \Spec\bigl(Y[K_q]\bigr)
 &=\{q\theta_i+q-1:0\le i\le q-2\}
   \cup\{(-1)^{(q-1)^2}\}.
\end{align*}
\end{lemma}

\begin{proof}
Order the vertices by the $q$-vertex fibers.  The two adjacency
matrices are
\[
 A_Y\otimes J_q
 \quad\text{and}\quad
 A_Y\otimes J_q+I_{q-1}\otimes(J_q-I_q),
\]
respectively.  The stated spectra follow by decomposing according to
the all-one vector and its orthogonal complement in each fiber.
\end{proof}

We use the following classical result.

\begin{theorem}[Doob--Cvetkovi\'c \cite{DoobCvetkovic1979}]
\label{thm:minus2}
A connected regular graph whose least adjacency eigenvalue is strictly
greater than $-2$ is either a complete graph or an odd cycle.
\end{theorem}

For a modern reference to this exact regular corollary, see
\cite[Corollary~2.3.22]{CvetkovicRowlinsonSimic2004}.

Write $\CP_r=K_{2r}\setminus M$ for the cocktail-party graph, where
$M$ is a perfect matching.

\begin{theorem}[Ramanujan $q$-fold blow-ups]\label{thm:graph}
Let $q\ge3$ and let $Y$ be a connected regular graph on $q-1$ vertices.
Up to coincidences in the smallest cases, the following hold.
\begin{enumerate}[label=\textup{(\alph*)}]
\item $Y[\overline{K_q}]$ is Ramanujan if and only if $Y$ is one of
\[
 K_{q-1},\qquad
 K_{r,r}\quad(q-1=2r),\qquad
 K_{r,r}\setminus M\quad(q-1=2r,\ q\ge7).
\]
\item $Y[K_q]$ is Ramanujan if and only if $Y$ is one of
\[
 K_{q-1},\qquad
 \CP_r=K_{2r}\setminus M\quad(q-1=2r,\ q\ge5).
\]
\end{enumerate}
\end{theorem}

\begin{proof}
We prove (a) first.  Suppose $Y[\overline{K_q}]$ is Ramanujan.  Its
degree is $d=qt$.

Assume first that $Y$ is non-bipartite.  Every nonprincipal eigenvalue
$\theta$ of $Y$ satisfies
\[
 q|\theta|\le2\sqrt{qt-1}<2q,
\]
so the least eigenvalue of $Y$ is greater than $-2$.  By
\cref{thm:minus2}, $Y$ is complete or an odd cycle.  The complete graph
gives the first family.  If $Y=C_{q-1}$ is a noncomplete odd cycle,
then $q\ge6$ and its least eigenvalue is
$-2\cos(\pi/(q-1))$.  Since
$\cos(\pi/(q-1))\ge\cos(\pi/5)>4/5$, we have
\[
 2q\cos\frac{\pi}{q-1}>2\sqrt{2q-1},
\]
so its blow-up is not Ramanujan.

Now suppose $Y$ is bipartite.  Write $q-1=2r$ and $t=r-s$.  The two
trivial eigenvalues are $t$ and $-t$.  Since
$\sum_i\theta_i^2=2rt$, removing these two terms gives
\[
 \sum_{\theta_i\ne\pm t}\theta_i^2=2t(r-t)=2ts.
\]
For $q=3$, connectedness forces $r=t=1$, hence $s=0$.  For $q\ge5$,
the Ramanujan bound and \cref{lem:blowupspectrum} give
\[
 2tsq^2\le4(q-3)(qt-1)<4(q-3)qt.
\]
Thus
\[
 qs<2(q-3),
\]
and hence $s\le1$.  If $s=0$, then $Y=K_{r,r}$.  If $s=1$, then every
vertex is missing exactly one neighbor on the opposite side, so
$Y=K_{r,r}\setminus M$.  This graph is connected only for $r\ge3$,
and its nontrivial eigenvalues are $\pm1$.  Its blow-up is Ramanujan
exactly when
\[
 q\le2\sqrt{\frac{q(q-3)}2-1},
\]
which is equivalent to $q\ge7$.  This proves the necessity in (a), and
the same spectral descriptions verify the sufficiency.

For (b), let $Z=\overline{Y}$, and write $m$ for the degree of $Z$.
Then $t=q-2-m$, while the degree upstairs is
\[
 d=q(q-1-m)-1.
\]
If $\alpha$ is an eigenvalue of $Z$ on the orthogonal complement of the
all-one vector, then the corresponding eigenvalue of $Y[K_q]$ is
\[
 -1-q\alpha.\tag{2.1}
\]
If $m=0$, then $Y=K_{q-1}$, giving the first family.  Suppose $m\ge2$.
If $Z$ is disconnected, then $m$ occurs as a nonprincipal eigenvalue,
and (2.1) has absolute value $qm+1$, which is larger than the
Ramanujan bound.  Hence $Z$ must be connected.

If the least eigenvalue of $Z$ is at most $-2$, then (2.1) produces an
eigenvalue at least $2q-1$, whereas
\[
 2\sqrt{d-1}\le2\sqrt{q(q-3)-2}<2q-1.
\]
Thus a Ramanujan blow-up forces $\lambda_{\min}(Z)>-2$.  By
\cref{thm:minus2}, $Z$ is complete or an odd cycle.  The complete case
makes $Y$ edgeless and hence disconnected.  A noncomplete odd cycle has
$m=2$ and $q\ge6$ even.  Its least eigenvalue gives upstairs
\[
 2q\cos\frac{\pi}{q-1}-1.
\]
For $q=6$ this is already larger than the bound $8$.  For $q\ge8$,
using $\cos x>1-x^2/2$ and
\[
 \frac{\pi^2q}{(q-1)^2}
 \le \frac{8\pi^2}{49}<2,
\]
we obtain
\[
 2q\cos\frac{\pi}{q-1}-1>2q-3.
\]
Moreover,
\[
 (2q-3)^2-4\bigl(q(q-3)-2\bigr)=17,
\]
so $2q-3>2\sqrt{q(q-3)-2}$.  Hence this case also fails, and we
conclude that $m\le1$.

If $m=1$, then $Z$ is a perfect matching, so $q$ is odd and
$Y=\CP_{(q-1)/2}$.  The largest absolute nontrivial eigenvalue upstairs
is $q+1$, and the Ramanujan inequality becomes
\[
 q+1\le2\sqrt{q(q-2)-2},
\]
which holds exactly for $q\ge5$.  Again the displayed spectra also
prove sufficiency.
\end{proof}

\begin{remark}
Theorem~\ref{thm:graph} is independent of Cayley graphs.  This is useful
conceptually: all group theory in the ratio-one problem is used only
to force one of the two blow-up forms; the spectral rigidity is then a
statement about arbitrary regular base graphs.
\end{remark}

\section{Ratio-one Frobenius groups}

\begin{definition}\label{def:frobenius}
A finite group $G$ is a \emph{Frobenius group} with kernel $N$ and
complement $H$ if $G=N\rtimes H$ is a semidirect product in which $H$
acts on $N\setminus\{1\}$ without fixed points, that is,
$C_N(h)=1$ for every $1\ne h\in H$.
\end{definition}

Let $G=N\rtimes H$ be such a group with
\[
 |N|=q,\qquad |H|=q-1,
\]
and let $\pi:G\to H$, $\pi(nh)=h$, be the quotient map.  We assume
$q\ge3$.  A \emph{normal Cayley connection set} is an inverse-closed
generating set $S\subset G\setminus\{1\}$ that is a union of conjugacy
classes.

\begin{lemma}\label{lem:classes}
Under the ratio-one hypothesis:
\begin{enumerate}[label=\textup{(\alph*)}]
\item $N\setminus\{1\}$ is one conjugacy class of $G$;
\item if $1\ne h\in H$, then
\[
 \operatorname{Conj}_G(h)=
 \pi^{-1}\bigl(\operatorname{Conj}_H(h)\bigr).
\]
\end{enumerate}
\end{lemma}

\begin{proof}
By \cref{def:frobenius}, $H$ acts fixed-point freely on
$N\setminus\{1\}$.  Its orbits therefore have size $|H|=q-1$, which
proves (a).

For $1\ne h\in H$, \cref{def:frobenius} gives $C_N(h)=1$.  Hence
$C_G(h)=C_H(h)$, and therefore
\[
 |\operatorname{Conj}_G(h)|
 =\frac{|G|}{|C_H(h)|}
 =q\,|\operatorname{Conj}_H(h)|.
\]
The left side maps under $\pi$ into the $H$-class of $h$, and the two
sets have the same cardinality, proving (b).
\end{proof}

\begin{proposition}\label{prop:form}
Every normal inverse-closed subset $S\subset G\setminus\{1\}$ is
uniquely of the form
\[
 S_{\varepsilon,E}
 =\varepsilon(N\setminus\{1\})\sqcup\pi^{-1}(E),
\]
where $\varepsilon\in\{0,1\}$ and
$E\subset H\setminus\{1\}$ is normal and inverse-closed.  Moreover,
$S_{\varepsilon,E}$ generates $G$ if and only if $E$ generates $H$.
\end{proposition}

\begin{proof}
The form follows from \cref{lem:classes}.  If $S_{\varepsilon,E}$
generates $G$, its image generates $H$.  Conversely, assume $E$
generates $H$.  Choose $h\in E$.  The whole coset $Nh$ lies in
$\pi^{-1}(E)$, so for every $n\in N$ both $h$ and $nh$ lie in the
generated subgroup, and
\[
 (nh)h^{-1}=n.
\]
Thus the subgroup generated by $S_{\varepsilon,E}$ contains $N$ and
maps onto $H$, hence equals $G$.
\end{proof}

\begin{proposition}[Forced blow-up]\label{prop:forced}
Let $S=S_{\varepsilon,E}$ be a normal Cayley connection set and put
$Y=\Cay(H,E)$.  Then
\[
 \Cay(G,S_{0,E})\cong Y[\overline{K_q}],\qquad
 \Cay(G,S_{1,E})\cong Y[K_q].
\]
\end{proposition}

\begin{proof}
Partition $G$ into the $q$-element fibers of $\pi$.  Between the fibers
indexed by $x,y\in H$, either every pair of vertices is adjacent or no
pair is adjacent, according as $x^{-1}y\in E$ or not.  Inside a fiber
there are no edges for $\varepsilon=0$ and all possible edges for
$\varepsilon=1$.
\end{proof}

\section{Classification in the Frobenius group}

We now translate \cref{thm:graph} back to connection sets.  Recall that
a connected Cayley graph on $H$ is bipartite if and only if its
bipartition containing $1$ is an index-two subgroup $K$ and its
connection set is contained in $H\setminus K$.

\begin{theorem}\label{thm:main}
Let $G=N\rtimes H$ be a finite Frobenius group with
$|N|=q\ge3$ and $|H|=q-1$.  A normal Cayley connection set
$S\subset G\setminus\{1\}$ gives a Ramanujan graph if and only if $S$
is of one of the following forms:
\begin{enumerate}[label=\textup{(\roman*)}]
\item $S=G\setminus\{1\}$;
\item $S=G\setminus N$;
\item $q\ge5$ and
\[
 S=(N\setminus\{1\})\cup
 \pi^{-1}\bigl(H\setminus\{1,z\}\bigr),
\]
where $z\in Z(H)$ is an involution (an element of order two);
\item
\[
 S=\pi^{-1}(H\setminus K),
\]
where $K\triangleleft H$ has index two;
\item $q\ge7$ and
\[
 S=\pi^{-1}\bigl((H\setminus K)\setminus\{z\}\bigr),
\]
where $K\triangleleft H$ has index two and
$z\in Z(H)\setminus K$ is an involution.
\end{enumerate}
Descriptions that coincide in the smallest cases are counted only
once.
\end{theorem}

\begin{proof}
Write $S=S_{\varepsilon,E}$ and $Y=\Cay(H,E)$.  By
\cref{prop:forced}, \cref{thm:graph} applies.

For $\varepsilon=0$, the complete base $K_{q-1}$ has
$E=H\setminus\{1\}$, giving (ii).  If the base is $K_{r,r}$, its
bipartition containing $1$ is an index-two subgroup $K$, and degree
considerations force $E=H\setminus K$, giving (iv).  If the base is
$K_{r,r}\setminus M$, then
\[
 E=(H\setminus K)\setminus\{z\}
\]
for one element $z$.  Since both $E$ and $H\setminus K$ are normal and
inverse-closed, the singleton $\{z\}$ is also normal and inverse-closed.
Thus $z$ is a central involution outside $K$, giving (v).

For $\varepsilon=1$, the complete base gives (i).  If the base is a
cocktail-party graph, its complement is a perfect matching and hence
the omitted connection set in $H$ has size one, say $\{z\}$.  Normality
and inverse-closure again make $z$ a central involution, giving (iii).

Conversely, the indicated base graphs are exactly those listed in
\cref{thm:graph}; the threshold conditions $q\ge5$ and $q\ge7$ are
those obtained there.  Their base graphs are connected, so the
corresponding sets generate $G$ by \cref{prop:form}.
\end{proof}

\begin{remark}
Hirano--Katata--Yamasaki already point out that small Frobenius ratios
can exhibit Ramanujan behavior beyond the generic valency boundary
\cite{HKY2016}.  The content of \cref{thm:main} is different: it gives
all normal connection sets in the extreme ratio-one case, not only the
first valency at which the generic bound can fail.
\end{remark}

\section{The affine group \texorpdfstring{$\AGL(1,q)$}{AGL(1,q)}}

We work out \cref{thm:main} for the most familiar family of
ratio-one Frobenius groups, the affine group of a finite field; other
examples are discussed at the end of the section.  Let
\[
 G=\AGL(1,q)=\mathbb F_q\rtimes\mathbb F_q^\times,
 \qquad H=\mathbb F_q^\times.
\]
The complement is cyclic of order $q-1$.  If $q$ is even, it has no
involution and no index-two subgroup, so only (i) and (ii) of
\cref{thm:main} occur.

Suppose $q$ is odd.  Then $H$ has the unique central involution $-1$
and the unique index-two subgroup
\[
 H^2=(\mathbb F_q^\times)^2.
\]
The fifth family occurs precisely when $-1\notin H^2$, equivalently
$q\equiv3\pmod4$.

\begin{corollary}\label{cor:agl}
Let $q\ge3$ be a prime power.  For $\AGL(1,q)$ the normal Ramanujan
connection sets comprise:
\begin{enumerate}[label=\textup{(\alph*)}]
\item two types if $q$ is even;
\item four types if $q\equiv1\pmod4$;
\item five types if $q\equiv3\pmod4$ and $q\ge7$.
\end{enumerate}
For $q=3$ there are two distinct graphs, namely $K_6$ and $K_{3,3}$.
\end{corollary}

In particular, one of the additional families missed by a
non-bipartite-only reading of the Ramanujan condition is
\[
 E=\mathbb F_q^\times\setminus(\mathbb F_q^\times)^2,
 \qquad \varepsilon=0.
\]
Its graph is simply
\[
 K_{q(q-1)/2,\,q(q-1)/2},
\]
and is therefore Ramanujan.  The fifth family, when
$q\equiv3\pmod4$, deletes the single multiplier $-1$ from this coset
and yields the blow-up of $K_{(q-1)/2,(q-1)/2}$ with a perfect matching
removed.

\begin{remark}[Beyond the affine case]\label{rem:nearfields}
The groups $\AGL(1,q)$ are only the cyclic-complement instance of
\cref{thm:main}: they are the affine groups $N\rtimes N^\times$ for
which the kernel $N$ is a field.  By the classification of finite
sharply two-transitive groups \cite{GrundhoferHering2017}, every
ratio-one Frobenius group has this form for some finite nearfield
$N$, and infinitely many nearfields are not fields.  The smallest
instance is the unique nearfield of order $9$ other than
$\mathbb F_9$; its multiplicative group is the quaternion group
$Q_8$, so $N\rtimes Q_8$, of order $72$, is a ratio-one Frobenius
group covered by \cref{thm:main} that is not isomorphic to any
$\AGL(1,q)$, since $Q_8$ is not cyclic.  Beyond this Dickson family
of nearfields there are also seven sporadic exceptional nearfields,
of orders $5^2,7^2,11^2,11^2,23^2,29^2,59^2$, each contributing a
further ratio-one Frobenius group to which \cref{thm:main} applies.
\end{remark}

\section{Concluding remarks}

The classification is driven by a hierarchy of increasingly coarse
structures.  Normality first reduces the connection set to a subset of
the complement; ratio one then upgrades that reduction to an exact
lexicographic product; finally the Ramanujan inequality forces the
small base graph into one of a handful of extremal spectral types.
This explains why a case-by-case Fourier analysis on the affine
complement eventually encounters the same few configurations.

The next natural question is what remains of this mechanism for
Frobenius ratios two and three, precisely the small ratios outside the
range of the general boundary theorem of Hirano--Katata--Yamasaki.
The ratio-one proof is unusually rigid because $N\setminus\{1\}$ is a
single conjugacy class.  For larger ratios it splits into several
classes, so one should expect a matrix-valued or multi-fiber analogue
rather than a single lexicographic blow-up.

\section*{Declaration of generative AI use}

Generative AI tools were used during manuscript preparation to assist
with drafting portions of the exposition and with mathematical reasoning.
All arguments and conclusions were independently checked by the authors,
who take full responsibility for the content.

\end{document}